\documentclass[preprint,12pt]{elsarticle}
\usepackage{amsthm,amsfonts,amssymb,amscd,amsmath,enumerate,verbatim,calc,graphicx,geometry}
\usepackage[all]{xy}
\newtheorem{theorem}{Theorem}[section]

\newtheorem{corollary}[theorem]{Corollary}
\theoremstyle{definition}
\theoremstyle{definitions}
\newtheorem{definition}[theorem]{Definition}

\newtheorem{remark}[theorem]{Remark}
\newtheorem{example}[theorem]{Example}
\theoremstyle{notations}

\theoremstyle{remarks}

\newcommand{\sub}{\subseteq}

\newcommand{\lo}{\longrightarrow}

\newcommand{\wt}{\widetilde}
\newcommand{\vf}{\varphi}

\newcommand{\al}{\alpha}
\newcommand{\la}{\lambda}

\newcommand{\bt}{\beta}
\newcommand{\ti}{\tilde}

\newcommand{\lk}{\langle}
\newcommand{\rg}{\rangle}
\newcommand{\psg}{\pi_1^{sg}(X,x)}

\newcommand{\pc}{p:\wt{X}\lo X}

\newcommand{\V}{\mathcal{V}}
\newcommand{\U}{\mathcal{U}}
\journal{Topology and its Applications}

\begin{document}

\begin{frontmatter}



\title{On the Spanier Groups and Covering and Semicovering Spaces}


\author[a]{Hamid~Torabi}
\ead{hamid$_{-}$torabi86@yahoo.com}
\author[b]{Ali~Pakdaman}
\ead{Alipaky@yahoo.com}
\author[a]{Behrooz~Mashayekhy\corref{cor1}}
\ead{bmashf@um.ac.ir}
\address[a]{Department of Pure Mathematics, Center of Excellence in Analysis on Algebraic Structures, Ferdowsi University of Mashhad,\\
P.O.Box 1159-91775, Mashhad, Iran.}
\address[b]{Department of Mathematics, Faculty of Sciences, Golestan University,\\
Gorgan, Iran.}
\cortext[cor1]{Corresponding author}
\begin{abstract}
For a connected, locally path connected space $X$, let $H$ be a subgroup of the fundamental group of $X$, $\pi_1(X,x)$. We show that there
exists an open cover $\cal U$ of $X$ such that $H$ contains the Spanier group $\pi({\U},x)$ if and only if the core of $H$ in $\pi_1(X,x)$
is open in the quasitopological fundamental group $\pi_1^{qtop}(X,x)$ or equivalently it is open in the topological fundamental group $\pi_1^{\tau}(X,x)$.
As a consequence, using the relation between the Spanier groups and covering spaces, we give a classification for connected covering spaces of $X$ based on
the conjugacy classes of subgroups with open core in $\pi_1^{qtop}(X,x)$. Finally, we give a necessary and sufficient condition for the existence of a semicovering. Moreover, we present a condition under which every semicovering of $X$ is a covering.
\end{abstract}


\end{frontmatter}


\section{Introduction and motivation}
The motivation of this paper is the following interesting classical theorem of Spanier \cite[Theorem 2.5.13]{S} for the existence of covering spaces:

\begin{theorem}
For a connected, locally path connected space $X$, if $H$ is a subgroup of $\pi_1(X,x)$ and there is an open cover $\U$ of $X$ such that $\pi(\U,x)\leq H$, then there exists a covering $\pc$ such that $p_*\pi_1(\wt{X},\ti{x})=H$.
\end{theorem}
Since for a locally path connected and semilocally simply connected space $X$ there exists an open cover $\U$ such that $\pi(\U,x)=1$, for a point $x\in X$, the existence of simply connected universal covering follows from the above theorem.

We recall from \cite{FRVZ} that the Spanier group $\pi(\U,x)$ with respect to the open cover $\mathcal{U}=\{U_i | i\in I\}$ is defined to be
the subgroup of $\pi_1(X, x)$ which contains all homotopy classes having representatives of the following type:
$$\prod\limits_{j=1}^{n}u_jv_ju_j^{-1},$$
where $u_j$ are arbitrary paths starting at $x$ and each $v_j$ is a loop inside one of the open sets $U_j\in\mathcal{U}$.

One of the main objects of this paper is to find conditions on a subgroup $H$ of $\pi_1(X,x)$ under which there is an open cover $\U$ of $X$ such that $\pi(\U,x)\leq H$.
In Section 2 without using Theorem 1.1, we show that every open normal subgroup $H$ of the quasitopological fundamental group $\pi_1^{qtop}(X,x)$ satisfies the above property. Then using this fact we obtain a weaker condition on the subgroup $H$ in order to satisfies the above property. In fact, we prove that for every subgroup $H$ of $\pi_1(X,x)$ with open core in $\pi_1^{qtop}(X,x)$ there is an open cover $\U$ of $X$ such that $\pi(\U,x)\leq H$. We also give an example to show that locally path connectedness of $X$ is essential for the above result.
Later in Section 3, we show that if the subgroup $H$ contains a Spanier group, then its core, $H_{\pi_1(X,x)}$, is open in $\pi_1^{qtop}(X,x)$.

We recall that the quasitopological fundamental group $\pi_1^{qtop}(X,x)$ is the quotient space of the loop space $\Omega(X,x)$ equipped with the compact-open topology
with respect to the function $ \Omega(X,x)\longrightarrow \pi_1(X,x)$ identifying path components (see \cite{B}). It should be mentioned that $\pi_1^{qtop}(X,x)$ is a quasitopological group in the sense of \cite{A} and it is not always a topological group (see \cite{Br2,F}). Also we recall from \cite{Br3} that the topological fundamental group $\pi_1^{\tau}(X,x)$ is the fundamental group $\pi_1(X,x)$
with the finest group topology on $\pi_1(X,x)$ such that the canonical function $ \Omega(X,x)\longrightarrow \pi_1(X,x)$ identifying path components is continuous. It should be mentioned that
$\pi_1^{\tau}(X,x)$ and $\pi_1^{qtop}(X,x)$ have the same open subgroups \cite[Proposition 4.4]{Br3} and every open set of $\pi_1^{\tau}(X,x)$ is also an open set in $\pi_1^{qtop}(X,x)$.

Biss \cite[Theorem 5.5]{B} showed that for a connected, locally path connected space $X$, there
is a 1-1 correspondence between its equivalent classes of connected covering spaces and the conjugacy classes of open subgroups of
its fundamental group $\pi_1(X,x)$. There is a misstep in the proof of the above theorem. In fact, Biss assumed that every fibration with discrete fiber is a covering map which is not true in general. We give an example (Example 4.2) to show that the above classification of connected covering spaces does not hold.
Recently, Brazas and Fabel \cite[Theorem 46]{BF} gave a classification for the equivalence classes of connected coverings of a connected, locally path connected, paracompact and Hausdorrf space $X$ based on the cojugacy classes of open subgroups of $\pi_1(X,x)$ with the shape topology.
In Section 3, first we introduce a {\it path open cover} $\V$ of a pointed space $(X,x)$ and also the {\it path Spanier group} $\wt{\pi}_(\V,x)$ with respect to $\V$. Then we show that $\wt{\pi}_(\V,x)$ is an open subgroup of $\pi_1^{qtop}(X,x)$ and using this fact we give a necessary and sufficient condition for a subgroup of $\pi_1(X,x)$ to be open in $\pi_1^{qtop}(X,x)$.  Second, using these facts and the main results of Section 2, we present a suitable true classification of connected covering spaces with respect to the Biss's ones and an extended classification with respect to the Brazas and Fabel's ones as follows which is the second main object of the paper.

{\it  For a connected, locally path connected space $X$, there
is a 1-1 correspondence between its equivalent classes of connected covering spaces and the cojugacy classes of subgroups of
its fundamental group $\pi_1(X,x)$, with open core in $\pi_1^{qtop}(X,x)$.}

Brazas \cite[Definition 3.1]{Br1} introduced the notion of  semicovering map as a local homeomorphism with continuous lifting of paths and homotopies. He \cite[Proposition 3.7]{Br1} showed that every covering is a semicovering but the converse is not true in general. In fact, he
\cite[Example 3.8]{Br1} showed that there is a semicovering and hence a Serre fibration of the Hawaiian earring with discrete fiber which is not a covering.
Brazas \cite[Corollary 7.20]{Br1} also gave a classification for connected semicoverings and asserts that for a connected, locally wep-connected space $X$ there is a Galois correspondence between its equivalence classes of connected semicovering spaces and the conjugacy classes of open subgroups of its topological fundamental group $\pi_1^{\tau}(X,x)$.
In Section 4 using this classification, we extend the well-known result of Spanier, Theorem 1.1, to connected semicoverings.

Brazas \cite[Corollary 7.2]{Br1} proved that for a connected, locally path connected and semilocally simply connected space $X$, the category of  all semicoverings of $X$, $\mathbf{SCov}(X)$ is equivalent to the category of all coverings of $X$, $\mathbf{Cov}(X)$. He also raised a question that if there are more general conditions guaranteeing $\mathbf{SCov}(X)\simeq \mathbf{Cov}(X)$.  The last object of the paper is to find weaker conditions than semilocally simply connectedness under which every semicovering is a covering. As the final main result of Section 4, we show that for a connected, locally path connected and semilocally small generated space $X$, every semicovering of $X$ is a covering space. We recall from \cite{T2} that a space $X$  is semilocally small generated if for every $x\in X$ there exists an open neighborhood $U$ of $x$ such that $i_*\pi_1(U,x)\leq \psg$, where $i_*$ is the induced homomorphism from the inclusion $i:U\hookrightarrow X$ and the small generated subgroup $\psg$ is the subgroup generated by the following set $$\{[\al*\bt*\al^{-1}]\ |\ [\bt]\in\pi_1^s(X,\al(1)),\ \al\in P(X,x)\},$$
where $P(X,x)$ is the space of all paths from $I$ into $X$ with initial point $x$ and $\pi_1^s(X,\alpha(1))$ is the small subgroup of $X$ at $\alpha(1)$ (see \cite{V}).


\section{Relation between open subgroups and Spanier groups}
The following theorem gives a sufficient condition on an open subgroup in order to contain a Spanier group.
\begin{theorem}
If $X$ is a locally path connected space and $H$ is an open normal subgroup of $\pi_1^{qtop}(X,x)$, then there exists an open cover $\U$ such that $\pi(\U,x)\leq H$.
\end{theorem}
\begin{proof}
 First we show that for every $y$ in the path component of $X$ containing $x$, there exists an open neighborhood $V$ of $y$ in $X$ such that $[\lambda](i_*\pi_1(V,y))[\lambda]^{-1}\leq H$ for all path $\lambda$ from $x$ to $y$, where $i_*$ is the induced homomorphism from the inclusion $i:V\hookrightarrow X$. To prove this, let $\alpha$ be a path from $x$ to $y$. Since $[\al*\al^{-1}]=[e_x] \in H$, $\al*\al^{-1}\in q^{-1}(H)$, where $q:\Omega(X,x)\lo\pi_1(X,x)$ is the quotient map defined by $q(\gamma)=[\gamma]$. Hence there exists a basic open neighborhood $\bigcap_{i=1}^n\lk K_i,U_i\rg$ of $\al*\al^{-1}$ in $\Omega(X,x)$ such that $\bigcap_{i=1}^n\lk K_i,U_i\rg\sub q^{-1}(H)$.

 Let $A=\{i\in\{1,2,...,n\}|\ 1/2\in K_i\}$ and $B=\{i\in\{1,2,...,n\}|\ 1/2\notin K_i\}$. Since $\bigcup_{i\in B}K_i$ is compact, there exists $\delta_1>0$ such that $[1/2-\delta_1,1/2+\delta_1]\cap K_i=\varnothing$ for every $i\in B$.
If $A\neq\varnothing$, then $V=\bigcap_{i\in A}U_i$ is a nonempty open subset of $X$ which contains $(\al*\al^{-1})(1/2) = \al(1) = y$. Choose $\delta_2>0$ such that $[1-2\delta_2,1]\sub \al^{-1}(V)$. If $A=\varnothing$, then put $V=X$ and $\delta_2=1/2$. Let $\bt:I\lo V$ be any loop at $y$. Define $f:I\lo X$ by
\begin{displaymath}
f(t)= \left\{
\begin{array}{lr}
\al(2t)                     &       0\leq t\leq 1/2 \\
\beta\circ\varphi_1(t)     &      1/2\leq t\leq 1/2+\delta/2\\
\al^{-1}(2\varphi_2(t)-1)   &  1/2+\delta/2\leq t\leq 1/2+\delta\\
\al^{-1}(2t-1)              &  1/2+\delta\leq t\leq 1,
\end{array}
\right.
\end{displaymath}
where $\delta=min\{\delta_1,\delta_2\}$, $\varphi_1:[1/2,1/2+\delta/2]\lo I$ and $\varphi_2:[1/2+\delta/2,1/2+\delta]\lo [1/2,1/2+\delta]$ are linear homeomorphisms with $\varphi_1(1/2)=0$ and $\varphi_2(1/2+\delta/2)=1/2$. By gluing lemma,
$f$ is continuous and hence it is a loop such that $[f]=[\al*\bt*\al^{-1}]$.

We claim that $f\in\bigcap_{i=1}^n\lk K_i,U_i\rg$.
For every $i\in B$, $f(K_i)\sub U_i$ since $K_i\sub I\setminus[1/2-\delta,1/2+\delta]$, $f|_{I\setminus[1/2,1/2+\delta]}=\al*\al^{-1}|_{ I\setminus[1/2,1/2+\delta]}$ and $\al*\al^{-1}(K_i)\sub U_i$.
If $A\neq\varnothing$, then for every $i\in A$, $f(K_i)\sub U_i$ since $f|_{I\setminus[1/2,1/2+\delta]}=\al*\al^{-1}|_{ I\setminus[1/2,1/2+\delta]}$, $\al*\al^{-1}(K_i)\sub U_i$ and $f([1/2,1/2+\delta])\sub V\sub U_i$.
Therefore $f\in\bigcap_{i=1}^n\lk K_i,U_i\rg\sub q^{-1}(H)$ which implies that $[\al*\bt*\al^{-1}]=[f]\in H$. Hence $[\al*\lambda^{-1}][\lambda*\bt*\lambda^{-1}][\lambda*\al^{-1}]=[\al*\bt*\al^{-1}]\in H$. Since $H$ is a normal subgroup, we have $[\lambda*\bt*\lambda^{-1}]\in [\al*\lambda^{-1}]^{-1}H[\al*\lambda^{-1}] = H$. Therefore $[\lambda](i_*\pi_1(V,y))[\lambda]^{-1}\leq H$.
Since $X$ is locally path connected, we can replace $V$ with a path connected neighborhood $V_y$ .

Put $\U = \{V_y|y \in X\}$. We show that $\pi(\U,x)\leq H$. let $[\al*\bt*\al^{-1}]$ be a generator of $\pi(\U,x)$, where $\bt$ is a loop in $V_y$ for some $y \in X$. Since $V_y$ is path connected, there is a path $\gamma$ in $V_y$ from $y$ to $\al(1)$. Hence $[\al*\bt*\al^{-1}]=[\gamma*\al^{-1}]^{-1}[\gamma*\bt*\gamma^{-1}][\gamma*\al^{-1}] \in H$ which implies that $\pi(\U,x)\leq H$.
\end{proof}

In the following example, we show that locally path connectedness is an essential hypothesis in Theorem 2.1.
\begin{example}
Let $X$ be a subspace of $\mathbb{R}^3$ which is obtained by taking the ``surface'' obtained by rotating the topologists' sine curve about its
limiting arc (see Figure. 1). Put $A=\{(0,0,s) \in \mathbb{R}^3 | -1\leq s \leq1\}$, then $Y=X\backslash A$ is connected, locally path connected and semilocally simply connected. Hence $\pi_1^{qtop}(Y,x)$ is discrete where $x=(\frac{1}{\pi},0,0)$. Therefore $\pi_1^{qtop}(X,x)$ is a discrete topological group which implies that $\{[c_x]\}$ is open. But it is easy to see that for every open cover $\U$ of $X$, $\pi(\U,x)$ is non-trivial (see \cite{FRVZ}).

 \begin{figure}
 \center
  \includegraphics[scale=0.3]{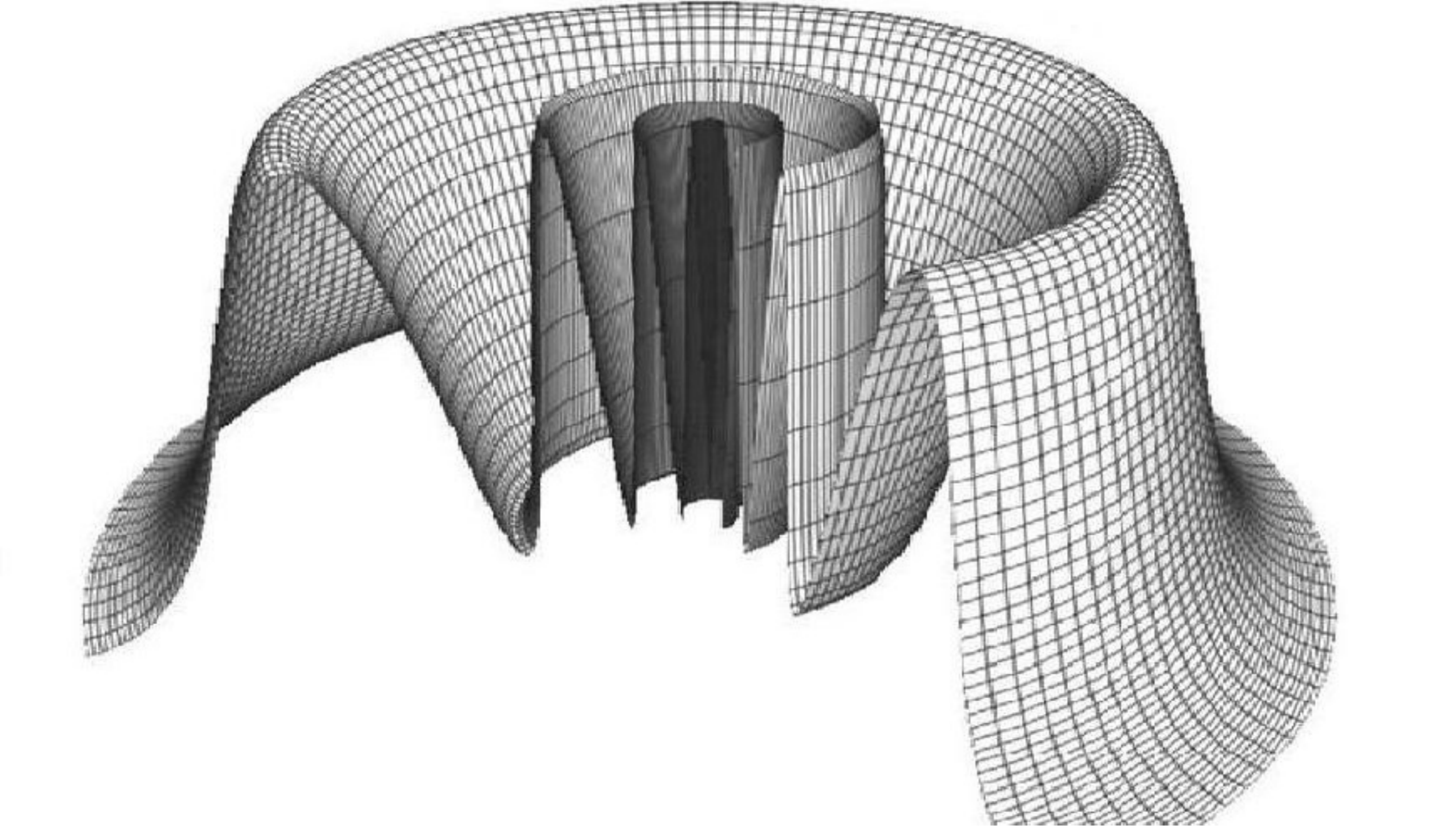}
  \caption{}\label{1}
\end{figure}

\end{example}
We recall from group theory that for any subgroup $H$ of a group $G$, the core of $H$ in $G$, denoted by $H_G$, is defined to be the join of all the normal subgroups of $G$ that are contained in $H$. It is easy to see that  $H_G=\bigcap_{g\in G}g^{-1}Hg$ which is the largest normal subgroup of $G$ contained in $H$.
Using this notion the following corollary is a consequence of Theorem 2.1.
\begin{corollary}
If $X$ is a locally path connected space and $H$ is a subgroup of $\pi_1(X,x)$ such that the core of $H$ in $\pi_1(X,x)$ is open in $\pi_1^{qtop}(X,x)$, then there exists an open cover $\U$ such that $\pi(\U,x)\leq H$.
\end{corollary}

\section{A classification of covering spaces}

In order to present a suitable classification of covering spaces we introduce the following concepts.
\begin{definition}
Let $(X,x)$ be a pointed space. By a {\em path open cover} of  $(X,x)$ we mean an open cover $\V=\{V_{\alpha} | \alpha \in P(X,x)\}$ of $X$ such that $\alpha(1) \in V_{\alpha}$ for every $\alpha \in P(X,x) $. We also define the {\em path Spanier group} $\wt{\pi}(\V, x)$ with respect to the path open cover $\V$ to be the subgroup of $\pi_1(X, x)$ which contains all homotopy classes having representatives of the following
type:
$$\prod\limits_{j=1}^{n}u_jv_ju_j^{-1},$$
where $u_j$ are arbitrary paths starting at $x$ and each $v_j$ is a loop inside $V_{u_j}$ for all $i \in\{1,2,...,n\}$.
\end{definition}

The following theorem gives an important property of the path Spanier groups.
\begin{theorem}
Let $(X,x)$ be a locally path connected pointed space and $\V=\{V_{\alpha} | \alpha \in P(X,x)\}$ be a path open cover of $(X,x)$. Then $\wt{\pi}(\V, x)$ is an open subgroup of $\pi_1^{qtop}(X,x)$.
\end{theorem}
\begin{proof}
It suffices to show that $q^{-1}(\wt{\pi}(\V, x))$ is an open subset of $\Omega(X,x)$, where $q:\Omega(X,x)\lo\pi_1(X,x)$ is the quotient map defined by $q(\al)=[\al]$. Let $\al\in q^{-1}(\wt{\pi}(\V, x))$. For every $t\in I$, put $\alpha_t = \alpha \circ \gamma$, where $\gamma:[0,1]\lo [0,t]$ is defined by $\gamma(s)=ts$ for every $s\in I$. The sets $\al^{-1}(V_{\alpha_t})$, $t\in I$, form an open cover of $I$. Since $I$ is compact and $t\in \al^{-1}(V_{\alpha_t})$ for every $t\in I$, there is a finite set $\{a_0=0 , a_1 , ... ,a_n=1\}\sub I $ such that for each $1\leq n\leq N$, $I_n=[a_{n-1},a_n]\sub \al^{-1}(V_{\alpha_{t_n}})$ for some $t_n\in I_n$. Note that for every $n\in\{1,2,...,N\}$, $[\alpha_{a_{n-1}}](i_*\pi_1(V_{\alpha_{t_n}},\alpha(a_{n-1})))[\alpha_{a_{n-1}}]^{-1} =[\alpha_{t_n}](i_*\pi_1(V_{\alpha_{t_n}},\alpha(t_n)))[\alpha_{t_n}]^{-1}\leq \wt{\pi}(\V, x)$. For every $1\leq n\leq N$, put $V_n = V_{\alpha_{t_n}}$ and for every $1\leq n\leq {N-1}$ consider $U_n$ to be the path component of $V_n\cap V_{n+1}$ containing $\al(a_n)$, so $$\al(a_n)\in U_n\sub V_n\cap V_{n+1}\sub X.$$
Since $X$ is locally path connected, $U_n$ is open. Consider the basic open set $$\mathcal{W}=\left(\bigcap_{n=1}^N\lk I_n,V_n\rg\right)\cap\left(\bigcap_{n=1}^{N-1}\lk\{a_n\},U_n\rg\right)\sub\Omega(X,x).$$
By construction, $\al\in\mathcal{W}$. It remains to show that $\mathcal{W}\sub q^{-1}(\wt{\pi}(\V, x))$. Let $\bt\in\mathcal{W}$, then
by definition of $\mathcal{W}$ we have $\bt(I_n)\sub V_n$, for each $1\leq n\leq N$, and $\bt(a_n)\in U_n$, for each $1\leq n\leq N-1$.
 For every $1\leq n\leq N$, put $\bt_n=\bt\circ\vf_n$, where $\vf_n:I\lo I_n$ is a linear homeomorphism such that $\vf_n(0)=a_{n-1}$. Also, for every $1\leq n\leq N-1$, let $\la_n:I\lo U_n$ be a path from $\al(a_n)$ to $\bt(a_n)$ by path connectivity of $U_n$.
By definitions of $\bt_n$ and $\la_n$'s, $\bt_1*\la_1^{-1}*\al_{a_1}^{-1}$ is a loop in $V_1$ based at $x$ and hence $[\bt_1*\la_1^{-1}*\al_{a_1}^{-1}]\in\wt{\pi}(\V, x)$ since $[\alpha_0](i_*\pi_1(V_1,\alpha(0)))[\alpha_0]^{-1}\leq \wt{\pi}(\V, x)$. Also, $\la_1*\bt_2*\la_2^{-1}*(\alpha\circ\vf_2)^{-1}$ is a loop in $V_2$ based at $\al(a_1)$ and hence $[\al_{a_1}*\la_1*\bt_2*\la_2^{-1}*\alpha_
{a_2}^{-1}]\in\wt{\pi}(\V, x)$.
Since $\wt{\pi}(\V, x)$ is a subgroup, we have $$[\bt_1*\bt_2*\la_2^{-1}*\al_{a_2}^{-1}]=[\bt_1*\la_1^{-1}*\al_{a_1}^{-1}][\al_{a_1}*\la_1*\bt_2*\la_2^{-1}*\alpha_
{a_2}^{-1}]\in\wt{\pi}(\V, x).$$
Similarly, $$g_1=[\beta_1*\beta_2*...*\beta_{N-1}*\la_{N-1}^{-1}*\alpha_{a_{N-1}}^{-1}]\in\wt{\pi}(\V, x).$$ Also, since $\la_{N-1}*\bt_N*(\alpha\circ\vf_{N-1})^{-1}$ is a loop at $\al(a_{N-1})$ in $V_N$, $g_2=[\al_{a_{N-1}}*\la_{N-1}*\bt_N*\alpha_1^{-1}]\in\wt{\pi}(\V, x)$.
Therefore
$$[\bt*\al^{-1}]=[\beta_1*\beta_2*...*\beta_{N-1}*\beta_N*\alpha_1^{-1}]=g_1.g_2\in\wt{\pi}(\V, x)$$
which implies that $[\bt]\in\wt{\pi}(\V, x)$ since $[\al]\in\wt{\pi}(\V, x)$.
\end{proof}

Using the above theorem, we can give a necessary and sufficient condition for a subgroup of $\pi_1(X,x)$ to be open in $\pi_1^{qtop}(X,x)$.
\begin{corollary}
 If $X$ is locally path connected, then for every subgroup $H$ of $\pi_1^{qtop}(X,x)$, $H$ is open if and only if there exists a path open cover $\V=\{V_{\alpha} | \alpha \in P(X,x)\}$ of $X$ such that $\wt{\pi}(\V, x)\leq H$.
\end{corollary}
\begin{proof}
Let $H$ be an open subgroup of $\pi_1^{qtop}(X,x)$. By the proof of Theorem 2.1, for every path $\alpha$ in $X$ from $x$ to any point $y$, there is an open set $V_{\alpha}$ of $\alpha(1)=y$ in $X$ such that $[\alpha](i_*\pi_1(V,y))[\alpha]^{-1}\leq H$. Hence by putting $\V=\{V_{\alpha} | \alpha \in P(X,x)\}$ as a path open cover of $X$ we have $\wt{\pi}(\V, x)\leq H$.
 The Converse follows from Theorem 3.2.
\end{proof}
\begin{remark}
Note that if $G$ is a quasitopological group and $H\leq K\leq G$ and $H$ is open in $G$, then $K$ is also an open subgroup of $G$ since every translation in $G$ is a homeomorphism.
\end{remark}

Now, we can show that the Spanier groups are open subgroups.
\begin{corollary}
If $X$ is a locally path connected space and $\U$ is an open cover of $X$, then $\pi(\U,x)$ is an open subgroup of $\pi_1^{qtop}(X,x)$.
\end{corollary}

\begin{proof}
For every $y\in X$ let $U_y$ be an element of cover $\U$ involve $y$. Consider $V_{\alpha} = U_y$ for every $\alpha \in P(X,x)$ with $\alpha(1)=y$. Hence $\wt{\pi}(\V, x) \leq \pi(\U,x)$ where $\V=\{V_{\alpha} | \alpha \in P(X,x)\}$. Therefore $\pi(\U,x)$ is open since $\wt{\pi}(\V, x)$ is an open subgroup.
\end{proof}

Now we are in a position to give a necessary and sufficient condition for a subgroup of the fundamental group $X$ to contain a Spanier group of $X$.
\begin{corollary}
If $X$ is locally path connected and $H$ is a subgroup of $\pi_1(X,x)$, then there exists an open cover $\U$ such that $\pi(\U,x)\leq H$ if and only if the core of $H$ in $\pi_1(X,x)$ is open in $\pi_1^{qtop}(X,x)$.
\end{corollary}
\begin{proof}
Let $\U$ be an open cover of $X$ such that $\pi(\U,x)\leq H$. Since $\pi(\U,x)$ is a normal subgroup of $\pi_1(X,x)$ we have $\pi(\U,x)\leq \bigcap_{g\in \pi_1(X,x)}g^{-1}Hg$. Hence by Remark 3.4, $H_{\pi_1(X,x)}=\bigcap_{g\in \pi_1(X,x)}g^{-1}Hg$ is an open subgroup of $\pi_1^{qtop}(X,x)$. The converse statement holds by Corollary 2.3.
\end{proof}

The following classification of connected coverings of locally path connected spaces is a consequence of Corollary 3.6 and Theorem 1.1 which is the main result of this section.
\begin{theorem}
For a connected, locally path connected space $X$, there
is a 1-1 correspondence between its equivalent classes of connected covering spaces and the cojugacy classes of subgroups of
its fundamental group $\pi_1(X,x)$, with open core in $\pi_1^{qtop}(X,x)$.
\end{theorem}
\begin{proof}
For a connected, locally path connected space $X$, if $H$ is a subgroup of $\pi_1(X,x)$ with open core, then  by Corollary 3.6 there exists an open cover $\U$ such that $\pi(\U,x)\leq H$. Thus by Theorem 1.1 there exists a covering $\pc$ such that $p_*\pi_1(\wt{X},\ti{x})=H$. Note that if there is another connected covering $q:\wt{Y}\rightarrow Y$, then by classical results in coverings, $q$ and $p$ are equivalent if and only if $q_*\pi_1(\wt{Y},\wt{y})=g^{-1}Hg$ for some $g\in \pi_1(X,x)$. Moreover, if there exists a covering $\pc$ such that $p_*\pi_1(\wt{X},\ti{x})=H$, then by choosing $\U$ consists of evenly covered open subsets of $X$ we have $\pi(\U,x)\leq H$. Hence by Corollary 3.6 the core of $H$ is open in $\pi_1^{qtop}(X,x)$.
\end{proof}

The following two corollaries are immediate consequences of the proof of Theorem 3.7.
\begin{corollary}
If $X$ is connected, locally path connected and $H$ is a subgroup of $\pi_1(X,x)$ with open core in $\pi_1^{qtop}(X,x)$, then there exists a covering $\pc$ such that $p_*\pi_1(\wt{X},\ti{x})=H$.
\end{corollary}
\begin{corollary}
If $X$ is connected,locally path connected space and there exists a covering $\pc$ such that $p_*\pi_1(\wt{X},\ti{x})=H$, then $H$ is open in $\pi_1^{qtop}(X,x)$.
\end{corollary}

Brazas and Fabel \cite[Theorem 46]{BF} proved that if $X$ is locally path connected, paracompact Hausdorrf, then there is a canonical bijection between the equivalence classes of connected coverings of $X$ and cojugacy classes of open subgroups of $\pi_1(X,x)$ with the shape topology. Using this fact and Theorem 3.7 we can obtain the following corollary.
\begin{corollary}
For a connected, locally path connected, paracompact and Hausdorrf space $X$, a subgroup $H$ of $\pi_1(X,x)$ is open with the shape topology if and only if the core of $H$ is open in $\pi_1^{qtop}(X,x)$.
\end{corollary}

\section{Semicovering Spaces}
Brazas \cite[Definition 3.1]{Br1} introduces the notion of  semicovering map as a local homeomorphism with continuous lifting of paths and homotopies.
 He \cite[Corollary 7.20]{Br1} gives a classification for connected semicoverings and asserts that for a connected, locally wep-connected space $X$ there is a Galois correspondence between its equivalence classes of connected semicovering spaces and the conjugacy classes of open subgroups of its topological fundamental group $\pi_1^{\tau}(X,x)$. Now, using the above classification of connected semicoverings and Corollary 3.3 we can extend the well-known result of Spanier, Theorem 1.1, to connected semicoverings.

\begin{theorem}
For a connected, locally path connected space $X$ and a subgroup $H$ of $\pi_1(X,x)$, there exists a semicovering $\pc$ such that $p_*\pi_1(\wt{X},\ti{x})=H$ if and only if there is a path open cover $\V$ of $X$ such that $\wt{\pi}(\V,x)\leq H$.
\end{theorem}

Brazas \cite[Proposition 3.7]{Br1} shows that every covering is a semicovering but the converse is not true in general. In fact he
\cite[Example 3.8]{Br1} shows that there is a semicovering and hence \cite[Remark 3.3]{Br1} a Serre fibration of Hawaiian earring with discrete fiber which is not a covering.
Using this example, we show that the classification of connected coverings given by Biss \cite[Theorem 5.5]{B} which we mentioned
in Section 1 does not hold.

\begin{example}
Let $X= \bigcup_{n\in \mathbb{N}}^{} \{(x,y) \in \mathbb{R}^2 |(x-\frac{1}{n})^2 + y^2 = \frac{1}{n^2}\}$ be the Hawaiian earing space. Brazas \cite[Example 3.8]{Br1} introduces a connected semicovering $\pc$ say and hence a Serre fibration of $X$ with discrete fiber which is not a covering. By classification of semicoverings \cite[Corollary 7.20]{Br1} $H = p_*\pi_1(\wt{X},\ti{x})$ is open in $\pi_1^{\tau}(X,x)$ and hence is open in $\pi_1^{qtop}(X,x)$. Now assume that the Biss's classification of connected coverings holds, then there exists a covering ${q:\wt{Y}\lo X}$ such that $p_*\pi_1(\wt{Y},\ti{y})=H$. Since every covering is a semicovering, ${q:\wt{Y}\lo X}$ is a semicovering which implies that $\pc$ and ${q:\wt{Y}\lo X}$ are equivalent
as semicoverings. Therefore $\pc$ is a covering which is a contradiction.
\end{example}

Brazas \cite[Corollary 7.2]{Br1} proves that for a connected, locally path connected and semilocally simply connected space $X$, the category of  all semicoverings of $X$, $\mathbf{SCov}(X)$ is equivalent to the category of all coverings of $X$, $\mathbf{Cov}(X)$ . He also raises a question that if there are more general conditions guaranteeing $\mathbf{SCov}(X)\simeq \mathbf{Cov}(X)$.  The last object of the paper is to find weaker conditions than semilocally simply connectedness under which every semicovering is a covering. First, we state the following note which seems interesting.
\begin{remark}
Since every covering is a semicovering and every semicovering equivalent to a covering as semicovering must be a covering, every semicovering of $X$ is covering if and only if for every open subgroup $H$ of $\pi_1^{qtop}(X,x)$, there exists a covering $\pc$ such that $p_*\pi_1(\wt{X},\ti{x})=H$ for some $\ti{x} \in \wt{X}$.
\end{remark}

Using the above remark we can prove the final main result of the paper.
\begin{theorem}
For a connected, locally path connected and semilocally small generated space $X$, every semicovering of $X$ is a covering.
\end{theorem}

\begin{proof}
Let $H$ be any open subgroup of $\pi_1^{qtop}(X,x)$. By Remark 4.3 it suffices to show that there exists a covering $\pc$ such that $p_*\pi_1(\wt{X},\ti{x})=H$. By Corollary 3.6 there is a path open cover $\V$ of $X$ such that $\wt{\pi}(\V,x)\leq H$. By definition of $\psg$ we can show that $\psg \leq \wt{\pi}(\V,x)$. Also by definition of semilocally small generatedness, there exists an open cover $\U$ of $X$ such that $\pi(\U,x)\leq \pi_1^{sg}(X,x)$. Therefore $\pi(\U,x)\leq H$. Now Theorem 1.1 gives the desired result.
\end{proof}

Finally, note that since $\pi_1^{\tau}(X,x)$ and $\pi_1^{qtop}(X,x)$ have the same open subgroups, we can replace $\pi_1^{qtop}(X,x)$ with
$\pi_1^{\tau}(X,x)$ in all of the results of the paper.













\end{document}